\documentclass[11pt]{article}
\usepackage{authblk}
\usepackage{mathrsfs}
\usepackage{amssymb}
\usepackage{amsmath}
\usepackage{mathrsfs}
\usepackage{amsthm}
\usepackage{amsfonts}
\usepackage{color}
\usepackage{latexsym}
\usepackage{txfonts}
\usepackage{indentfirst}
\usepackage{soul}
\usepackage[scriptsize,nooneline]{subfigure}
\usepackage[colorlinks, citecolor=red]{hyperref}
\usepackage[T1]{fontenc} 
\usepackage{mdwlist}
\numberwithin{equation}{section}
\allowdisplaybreaks[3]

\newtheorem{theorem}{Theorem}[section]
\newtheorem{lemma}{Lemma}[section]
\newtheorem{definition}{Definition}[section]

\newtheorem{remark}{Remark}[section]

\def\geq{\geqslant}
\def\leq{\leqslant}
\def\l{\left}
\def\r{\right}
\def\g{\mathfrak g}
\def\h{\mathfrak h}

\begin{document}
	\title {Weak Factorizations of the Hardy Spaces in Terms of Multilinear Calder\'on-Zygmund Operators on Ball Banach Function Spaces}
	\author[1]{Yichun Zhao}
	\author[2]{Xiangxing Tao}
	\author[1]{Jiang Zhou \thanks{Corresponding author E-mail: zhoujiang@xju.edu.cn. The research was supported by National Natural Science Foundation of China (Grant Nos. 12061069 and 12271483).}}
	\affil[1]{College of Mathematics and System Sciences, Xinjiang University, Urumqi 830046, PR China.}
	\affil[2]{Department of Mathematics, Zhejiang University of Science and Technology, Hangzhou 310023, PR China.}
	
	
	\renewcommand*{\Affilfont}{\small\it} 
	\renewcommand\Authands{ and } 
	\date{}
	\maketitle
	{\bf Abstract:}{~In this paper, our main purpose is to establish a weak factorization of the classical Hardy spaces in terms of a multilinear Calder\'on-Zygmund operator on the ball Banach function spaces.  Furthermore, a new characterization of the BMO space via the boundedness of the commutator generated by the multilinear Calder\'on-Zygmund operator is also obtained.  The results obtained in this paper have generality.  As examples, we apply the above results to weighted Lebesgue space, variable Lebesgue space, Herz space, mixed-norm Lebesgue space, Lorentz space and so on.
	}\par
	
	
	\baselineskip 15pt
	\section{Introduction}
	Hardy space $H^1(\mathbb R^n)$ is crucial in harmonic analysis. In the study of operator theory, the real-variable space $H^1(\mathbb R^n)$ is seen as a suitable alternative space to the space $L^1(\mathbb R^n)$, as a result of which, the Hardy space has received widespread attention and has been systematically studied. In particular, the characterization of Hardy spaces is a key part of Hardy space theory. For example, atomic decompositions can be used to investigate the duality of Hardy spaces \cite[the Section 3.2]{grafakos2009modern}, and molecular decompositions can effectively simplify the proof of the boundedness of operators on Hardy spaces \cite[the Section 2.4]{grafakos2009modern}. Both of the above decompositions use some functions with good integrability as the basic building blocks.  While, in 1976, an important work, the weak factorization of Hardy spaces was established by Coifman, Rochberg and Weiss \cite{coifman1976factorization}, who proved that functions in Hardy spaces could be represented by a linear combination of bilinear operators that was generated by the Riesz transforms, i.e. for any $f\in H^1(\mathbb R^n)$ can be represented that
	$$
	f=\sum_{i = 1}^\infty \sum_{j=1}^n \Pi_l\left(g_j^i, h_{j}^i\right)=\sum_{i = 1}^\infty  \sum_{j=1}^n \l[g_{ j}^i \mathcal{R}_j\left(h_{j}^i\right)+h_{j}^i \mathcal{R}_j\left(g_{j}^i\right)\r]~~
	\text{with}~~
	\sum_{i = 1}^\infty  \sum_{j=1}^n\left\|g_{ j}^i\right\|_{L^2}\left\|h_{j}^i\right\|_{L^2} \lesssim \|f\|_{\mathbf{H}^1}
	$$
	where $\mathcal{R}_j$ are the Riesz transforms, for $j=1, \ldots, n$.  In fact, as an application, BMO spaces can be characterized by the boundedness of the commutator generated by the Riesz transforms and functions in BMO spaces. After that, the results for weak factorizations of the space $H^1(\mathbb R^n)$ have been emerging, and more works can refer to \cite{duong2017factorization, komori2006factorization, li2017weak, uchiyama1981factorization, wang2023weak, wang2023factorization}.
	
	The multilinear Calder\'on-Zygmund operator is regarded as a generalization of the Riesz transform. Therefore the weak factorization of Hardy spaces via the multilinear Calder\'on-Zygmund operator on different spaces has also attracted much attention. In 2006, Komori and Mizuhara \cite{komori2006factorization} established a weak factorization of Hardy spaces on generalized Morrey spaces. In recent years, Li and Wick \cite{li2017weak} decomposed functions in the space $H^1(\mathbb R^n)$ by multilinear Calder\'on-Zygmund operators on Lebesgue spaces and given a new characterization of BMO space. Furthermore, Zhu and Wang \cite{wang2023factorization} generalized Li and Wick's results in \cite{li2017weak} to weighted Lebesgue spaces. Inspired by the results in \cite{wang2023factorization}, He and Tao \cite{he2023factorization} obtained a weak factorization of weighted Hardy spaces on the same space in 2019. Dao and Wick \cite{dao2023hardy} built the weak factorization of the Hardy space on Morrey spaces, which can be seen as natural generalizations of Lebesgue spaces.
	
	In 2017, Sawano et al. \cite{sawano2017hardy} defined the ball Banach space, which can unify many classical spaces in the harmonic analysis, such as Lebesgue space, weighted Lebesgue space, Herz space, mixed-paradigm space, Lorentz space and so on.   It is worth noting that compared to the Banach function space, the ball Banach function space contains the Morrey space. For more details about ball Banach function space, we can refer to \cite{wei2023block, zhang2021weak, tao2021compactness}. 
	
	Our main objective of this paper is to establish a weak factorization of Hardy spaces using a multilinear Calder\'on-Zygmund operator on the ball Banach function space and to further address the characterization of BMO spaces via commutator generated by multilinear Calder\'on-Zygmund operator and the BMO function. The weak factorization of Hardy spaces in this paper unifies the results in the \cite{komori2006factorization, li2017weak, wang2023factorization, he2023factorization, dao2023hardy}. As examples, we apply the results in this paper to weighted Lebesgue space, variable Lebesgue space, Herz space, mixed-norm Lebesgue space, and Lorentz space.
	
	Throughout the paper, we use the symbol  $C$ to denote a positive constant. The notations $\chi_E$  and $|E|$ mean that characteristic function and Lebesgue measure of measurable set $E$,  respectively. Moreover, for any $p\in [1,\infty]$, we write the conjugate index of $p$ as $p^\prime$ with $1/q+1/q^\prime =1$. We will write $A\lesssim B$ stands for $A\leq C B$. Furthermore, the $A\lesssim B\lesssim A$ will be denoted by $A\approx B$. In what follows, {\rm $\mathcal{M}$} stands for the set of all measurable functions on $\mathbb{R}^n$.  We also denote by $\mathbb B$ the set $\mathbb B=\{B(x,r): x\in \mathbb R^n, 0<r<\infty\}$, where $B(x,r)=\{y\in\mathbb R^n: |x-y|< r\}$. The notion $\mathcal{S}\left(\mathbb{R}^n\right)$ means that the spaces of all Schwartz functions on $\mathbb{R}^n$. We will use the symbol $\mathcal{S}^{\prime}\left(\mathbb{R}^n\right)$ to denote the dual space of $\mathcal{S}\left(\mathbb{R}^n\right)$. We also denote by $f_B$ the integral mean of the function $f$.

	\section{Preliminaries and  Statement of Main Results }
	\subsection{Ball Banach Function Spaces}
	To state the key theorems in this paper accurately and briefly, we first recall the definition of ball Banach function spaces and their related properties in this section.
	\begin{definition}\rm {\cite[Defination~2.2]{sawano2017hardy}}\label{ball Banach Def}
		We say a Banach space $X \subset \mathcal{M}\left(\mathbb{R}^n\right)$ is a ball Banach function space if the norm of $X$ satisfies the following properties:
		\begin{itemize*}
			\item [(1)]~$\|f\|_X=0$ if and only if $f=0$ almost everywhere;
			\item [(2)]~$|g| \leq|f|$ almost everywhere implies that $\|g\|_X \leq\|f\|_X$;
			\item [(3)]~$0 \leq f_m \nearrow f$ almost everywhere implies that $\left\|f_m\right\|_X \nearrow\|f\|_X$;
			\item [(4)]~for any ball $B \subset \mathbb{B}$, we have $\l\|\chi_B\r\|_X<\infty$;
			\item [(5)]~for any ball $B \subset \mathbb{B}$, there exists a positive constant $C_B$, depending on $B$, such that, for all $f \in X$,
			$$
			\int_B|f(x)| d x \leq C_B\|f\|_X .
			$$
		\end{itemize*}
	\end{definition}
	
	\begin{remark}\rm \label{re Minkowski}
		The Definition \ref{ball Banach Def} points that the norm of ball Banach function space $X$ satisfies the triangle inequality, i. e. for any $f,g\in X$, the following inequality holds
		\begin{align}\label{ine 2.1}
			\|f+g\|_X \leq\|f\|_X+\|g\|_X.
		\end{align}
	\end{remark}
	
	The inequality \eqref{ine 2.1} in Remark \ref{re Minkowski} can be seen as the Minkowski inequality when ball Banach function spaces $X$ replace to the Lebesgue spaces $L^p(\mathbb R^n)$. To give the corresponding H\"older inequality on ball Banach spaces, we will recall the associate space of the ball Banach spaces.
	\begin{definition}\rm {\cite[Defination~2.3]{sawano2017hardy}}
		The associate space (K\"othe dual) of ball Banach function space $X$ is defined by
		$$
		X^{\prime}:=\left\{f \in \mathcal{M}\left(\mathbb{R}^n\right):\|f\|_{X^{\prime}}=\sup _{\|g\|_X=1} \int_{\mathbb{R}^n}|f(x) g(x)| d x<\infty\right\},
		$$
		and the norm $\|\cdot\|_{X^{\prime}}$ can be seen as the norm of associate space $X^\prime$.
	\end{definition}
	\begin{remark}\rm \label{re ball Banach}
		Let $X$ be a ball Banach function space with associate spaces $X^\prime$. The following properties were proven by Zhang et al. in \cite{zhang2021weak}.
		\begin{enumerate*}
			\item [(1)]~The associate space $X^\prime$ is also
			a ball Banach function space.
			\item [(2)]~For any $1\leq t<\infty$, the spaces $X^t$ is also a ball Banach function space.
			\item [(3)]~The H\"older inequality on ball Banach space, for any $f\in X$, $g\in X^\prime$, we deduce that  $fg\in L^1(\mathbb R^n )$, i.e.
			\begin{align}\label{ine Holder}
				\int_{\mathbb{R}^n}|f(x) g(x)| d x \leq\|f\|_X\|g\|_{X^{\prime}}.
			\end{align} 
		\end{enumerate*}
	\end{remark}
	
	\subsection{ Multilinear Calder\'on-Zygmund operators}
	\begin{definition}\rm \label{def Calderon}
		Let $K\left(z_0, z_1, \ldots, z_m\right), z_i \in \mathbb{R}^n ~\text{for all} ~i=0,1, \ldots, m$, be a locally integrable function, defined away from the diagonal $\left\{z_0=z_1=\cdots=z_m\right\}$. Then, the function $K$ is called an m-linear Calder\'on-Zygmund kernel if it satisfies the following size condition
		\begin{align*}
			\left|K\left(z_0, z_1, \ldots, z_m\right)\right| \lesssim \frac{1}{\left(\sum_{k, l=0}^m\left|z_k-z_l\right|\right)^{m n}},
		\end{align*}
		and, for some $\delta>0$, the smoothness condition holds
		\begin{align}\label{smooth condition}
			\left|K\left(z_0, z_1, \ldots, z_j, \ldots, z_m\right)-K\left(z_0, z_1, \ldots, z_j^{\prime}, \ldots, z_m\right)\right| \lesssim \frac{\left|z_j-z_j^{\prime}\right|^\delta}{\left(\sum_{k, l=0}^m\left|z_k-z_l\right|\right)^{m n+\delta}}
		\end{align}
		whenever $\left|z_j-z_j^{\prime}\right| \leq \frac{1}{2} \max _{0 \leq k \leq m}\left|z_j-z_k\right|$.
		
		Let $T$ be an $m$-linear operator mapping from $\mathcal{S}\left(\mathbb{R}^n\right) \times \mathcal{S}\left(\mathbb{R}^n\right) \times \cdots \times \mathcal{S}\left(\mathbb{R}^n\right)$ to $\mathcal{S}^{\prime}\left(\mathbb{R}^n\right)$. Then, the $T$ associated with the $m$-linear Calder\'on-Zygmund kernel $K$ is defined by 
		$$
		T\left(f_1, \ldots, f_m\right)(x)=\int_{\mathbb{R}^{m n}} K\left(x, z_1, \ldots, z_m\right) \prod_{j=1}^m f_j\left(z_j\right) d z_1 \cdots d z_m,
		$$
		whenever $f_1, \ldots, f_m \in \mathcal{S}\left(\mathbb{R}^n\right)$ with compact support and $x \notin \cap_{j=1}^m \operatorname{supp}\left(f_j\right)$. 
		
		Furthermore, we say that the $m$-linear Calder\'on-Zygmund operator $T$ is $m n$-homogeneous, if for any $x\in B_0\left(x_0, r\right)$,the following inequality holds
		\begin{align}\label{ine mn-homogeneous}
			\left|T\left(\chi_{B_1}, \ldots, \chi_{B_m}\right)(x)\right| \geq \frac{1}{M^{m n}}, \quad \forall x \in B_0\left(x_0, r\right)
		\end{align}
		where $B_i=B\left(x_i, r\right)$, $i=0, 1, \dots, m$ satisfy the condition that $\left|z_0-z_l\right| \approx M r$ for all $z_0 \in B_0$, and $z_l \in B_l, l=1,2, \ldots, m$, where $r>0$ and large enough $M$.
	\end{definition}
	
	To define the bilinear operators generated by the $m$-linear Calder\'on-Zygmund operator $T$, we say the the $T_j^*$ is $j$-th transpose of $T$, if, for any $f_1, \ldots, f_m, g \in \mathcal{S}$, we have
	\begin{align}\label{ine T*}
		\left\langle T_j^*\left(f_1, \ldots, f_m\right), g\right\rangle=\left\langle T\left(f_1, \ldots, f_{j-1}, g, f_{j+1}, \ldots, f_m\right), f_j\right\rangle.
	\end{align}
	
	\begin{definition}\rm 
		Let $T$ be an operator associated with the $m$-linear Calder\'on-Zygmund kernel as Definition \ref{def Calderon}. We define the multilinear operators associated with $T$ by
		$$
		\Pi_l\left(\h, \g_1, \ldots, \g_m\right)(x):=\g_l(x) T_l^*\left(\g_1, \ldots, \g_{l-1}, \h, \g_{l+1}, \ldots, \g_m\right)(x)-\h(x) T\left(\g_1, \ldots, \g_m\right)(x) .
		$$
	\end{definition}
	
	Let $T$ be an operator associated with the $m$-linear Calder\'on-Zygmund kernel as Definition \ref{def Calderon}. The $l-$ th partial multilinear commutator of $T$ can be defined by 
	\begin{align*}
		[b, T]_{l}\left(\g_{1}, \ldots, \g_{m}\right)(x):=T\left(\g_{1}, \ldots, b \g_{l}, \ldots, \g_{m}\right)(x)-b T\left(\g_{1}, \ldots, \g_{m}\right)(x) .
	\end{align*}
	\subsection{The Atomic Hardy Spaces $H^1(\mathbb R^n)$}
	\begin{definition}\rm \label{def H1}
		We say a real-valued function $a$ is an atom, if the function $a$ satisfies the following conditions
		$$
		{(\rm i)}\operatorname {\rm supp}(a)\subset B(x, r) \subset \mathbb B,\quad  {(\rm ii)}\int_{\mathbb{R}^n} a(x) d x=0, \quad {(\rm iii)}~\|a\|_{L^{\infty}} \leq r^{-n} .
		$$
		Furthermore, we define the atomic Hardy space $H^1(\mathbb R^n)$ by
		$$
		\mathbf{H}^1\left(\mathbb{R}^n\right)=\left\{f: f=\sum_{i = 1}^\infty \lambda_i a_i,~ a_i~\text{is an atom for any }~i=1, 2, \dots, \{\lambda_i\}\in \ell^1(\mathbb R)\right\} .
		$$	
		Moreover, we define the norm of the atomic Hardy space $\mathbf{H}^1\left(\mathbb{R}^n\right)$ as follows.
		$$
		\|f\|_{\mathbf{H}^1\left(\mathbb{R}^n\right)}=\inf \left\{\sum_{i = 1}^\infty\left|\lambda_i\right|: f=\sum_{i = 1}^\infty \lambda_i a_i\right\} ,
		$$
		where the infimum is taken over all possible representations of $f$
	\end{definition}
	\subsection{Main Results}
	To provide a brief statement of the theorem in this paper, unless otherwise stated, we assume that $T$ is an m-linear Calderón-Zygmund operator with $m n$-homogeneous condition \eqref{ine mn-homogeneous} and the means of the infimum $\inf$ is taken over all possible decompositions of $f$ that satisfy \eqref{ine represent} in this paper.
	\begin{theorem}\label{the factorization}
		Let $1 \leq l \leq m$, $1<p_i<\infty$ for all $i=1, 2, \dots, m$ and $1/p_0=1/p_1+1/p_2+\dots+1/p_m$ with $1\leq p_0<\infty$. Assume the function $b\in \rm BMO$, the operator $T$ and commutators $[b,T]_l$ are bounded from $X^{p_1}\times \dots \times X^{p_m}$ to $X^{p_0}$. Then, for any $f \in \mathbf{H}^1\left(\mathbb{R}^n\right)$, there exists sequences $\left\{\lambda_{i,j}\right\} \in \ell^1$ and functions $\h_j^i\in [X^{p_0}]^\prime$ and $\g_{j, 1}^i\in X^{p_1}, \ldots, \g_{j, m}^i \in X^{p_m}$ such that
		\begin{align}\label{ine represent}
			f=\sum_{i=1}^{\infty} \sum_{j=1}^{\infty} \lambda_{i,j} \Pi_l\left(\h_j^i, \g_{j, 1}^i, \ldots, \g_{j, m}^i\right)~~~~~~\text{in the sense of ~}~\mathbf{H}^1\left(\mathbb{R}^n\right)
		\end{align}
		Moreover, the following equivalence of norm holds
		$$
		\|f\|_{\mathbf{H}^1} \approx \inf \left\{\sum_{i=1}^{\infty} \sum_{j=1}^{\infty}\left|\lambda_{i,j}\right|\left\|\h_j^i\right\|_{[X^p]^\prime}\left\|\g_{j, 1}^i\right\|_{X^{p_1}} \dots \l\|\g_{j, m}^i\r \|_{X^{p_m}} \right\}.
		$$
	\end{theorem}
	
	\begin{theorem}\label{the characterization BMO}
		Let $1 \leq l \leq m$, $1<p_i<\infty$ for all $i=1, 2, \dots, m$ and $1/p_0=1/p_1+1/p_2+\dots+1/p_m$ with $1\leq p_0<\infty$. Assume the function $b\in \rm BMO$, the operator $T$ and commutators $[b,T]_l$ are bounded from $X^{p_1}\times \dots \times X^{p_m}$ to $X^{p_0}$. Then, the commutator $[b,T]_l$ is bounded from $X^{p_1}\times \dots \times X^{p_m}$ to $X^{p_0}$ if and only if $b\in {\rm BMO}$, i.e.
		\begin{align*}
			b\in {\rm BMO}(\mathbb R^n)=\l\{f\in L_{\rm loc}^1:\|f\|_{\rm BMO}=\sup_{B\in \mathbb B}\frac{1}{|B|}\int_B|f(x)-f_B| dx<\infty\r\}.
		\end{align*}
	\end{theorem}
	
	\section{The proofs of Theorem \ref{the factorization} and Theorem \ref{the characterization BMO}}
	In this section, we will establish the weak Hardy factorization in terms of multilinear commutators $[b,T]_l$ on ball Banach function spaces, Furthermore, a new characterization of $\rm BMO$ space is obtained. To prove the theorems in this paper, we first establish some key lemmas as follows.
	\begin{lemma}{\rm \cite{li2017weak}} \label{lem log}\rm 
		Suppose $M$ is a enough large number, and $y, z \in \mathbb{R}^n$ such that $|y-z|=Mr$. If the function $\mathcal H$ satisfies that 
		\begin{align*}
			{(\rm i)} \int_{\mathbb{R}^n} \mathcal H(x) d x=0~~~~\text{,}~~~~	{(\rm ii)}~|\mathcal H(x)| \leq r^{-n}\left(\chi_{B\left(y, r\right)}(x)+\chi_{B\left(z, r\right)}(x)\right),
		\end{align*}
		for any $x \in \mathbb{R}^n$ and some $r>0$. Then, 
		$$
		\|\mathcal H\|_{\mathbf{H}^1\left(\mathbb{R}^n\right)} \lesssim \log M .
		$$
	\end{lemma}
	\begin{lemma}\label{lem proud_lH1}
		Let $1 \leq l \leq m$, $1<p_1,p_2,\dots, p_m<\infty$ and $1/p_0=1/p_1+1/p_2+\dots+1/p_m$ with $1\leq p_0<\infty$. Assume the function $b\in \rm BMO$, the operators $T$ and commutators $[b,T]_l$ are bounded from $X^{p_1}\times \dots \times X^{p_m}$ to $X^{p_0}$. Then, for any $\h\in [X^{p_0}]^\prime$ and $\g_{1}\in X^{p_1}, \ldots, \g_{m}\in X^{p_m}$, the following estimate holds
		$$
		\left\|\Pi_l\left(\h, \g_1, \ldots, \g_m\right)\right\|_{\mathbf H^1} \leqslant C\|\h\|_{ [X^{p_0}]^\prime}\left\|\g_1\right\|_{X^{p_1}} \ldots\left\|\g_m\right\|_{X^{p_m}} .
		$$
	\end{lemma}
	\begin{proof}
		For any $\h\in [X^{p_0}]^\prime$ and $\g_{1}\in X^{p_1}, \ldots, \g_{m}\in X^{p_m}$, from the boundedness of $T$, one obtains that 
		\begin{align*}
			\int_{\mathbb{R}^n}\left|T\left(\g_1, \ldots, \g_m\right)(x) \h (x) \right| d x 
			\leqslant\|\h\|_{[X^{p_0}]^\prime}\left\|T\left(\g_1, \ldots, \g_m\right)\right\|_{X^{p_0}} 
			\lesssim \|\h\|_{[X^{p_0}]^\prime}\prod_{j=1}^m\left\|\g_i\right\|_{X^{p_i}} .
		\end{align*}
		We conclude from $1/p_0=1/p_1+1/p_2+\dots+1/p_l+\dots+1/p_m$ that
		\begin{align}\label{ine indexs}
			\frac{1}{p_l^\prime}=\frac{1}{p_1}+\frac{1}{p_2}+\dots+\frac{1}{p_0^\prime}+\dots+\frac{1}{p_m}.
		\end{align}
		By the definition of $(T^*)_l$ \eqref{ine T*} and the equality \eqref{ine indexs}, we know the  operator $(T^*)_l$ also bounded from $X^{p_1}\times \dots \times X^{p_m}$ to $X^{p_0}$. The same proof works for $(T^*)_l$, which implies that  $\Pi_l\left(\h, \g_1, \ldots, \g_m\right)(x) \in L^1$.\\
		Furthermore, the direct calculation and the equality \eqref{ine T*} tell us that
		\begin{align*}
			&\quad \int_{\mathbb{R}^n} \Pi_l\left(\h, \g_1, \ldots, \g_m\right)(x) d x\\
			&=\int_{\mathbb{R}^n} \g_l(x) T_l^*\left(\g_1, \ldots, \g_{l-1}, \h, \g_{l+1}, \ldots, \g_m\right)(x) dx-\int_{\mathbb{R}^n}\h(x) T\left(\g_1, \ldots, \g_m\right)(x) dx =0 .
		\end{align*}
		Since the duality of Hardy space, it follows that 
		\begin{align*}
			\left\|\Pi_l\left(\h, \g_1, \ldots, \g_m\right)\right\|_{\mathbf H^1}\leq \sup_{\|b\|_{\rm BMO}\not= 0}\frac{\left|\int_{\mathbb{R}^n} b(x) \Pi_l\left(\h, \g_1, \ldots, \g_m\right)(x) d x\right|}{\|b\|_{\rm BMO}}.
		\end{align*}
		Hence, it suffices to show that for any $b \in \operatorname{BMO}$, we have
		\begin{align*}
			\left|\int_{\mathbb{R}^n} \Pi_l\left(\h, \g_1, \ldots, \g_m\right)(x)b(x) d x\right|=\left|\int_{\mathbb{R}^n} [b, T]_l\left(\g_1, \ldots, \g_m\right)(x)  \h(x) d x\right|\lesssim \|\h\|_{ [X^{p_0}]^\prime}\left\|\g_1\right\|_{X^{p_1}} \ldots\left\|\g_m\right\|_{X^{p_m}}\|b\|_{\mathrm{BMO}}.
		\end{align*}
		Therefore, from the H\"older inequality and the boundedness of $[b, T]_l$, we see that, for any $b \in \operatorname{BMO}$,
		\begin{align*}
			&\left|\int_{\mathbb{R}^n} [b, T]_l\left(\g_1, \ldots, \g_m\right)(x) \h(x)d x\right| \\
			& \leqslant \left\|[b, T]_l\left(\g_1, \ldots, \g_m\right)\right\|_{X^{p_0}} \|\h\|_{ [X^{p_0}]^\prime}\\
			& \lesssim \left\|\g_1\right\|_{X^{p_1}} \ldots\left\|\g_m\right\|_{X^{p_m}}\|\h\|_{ [X^{p_0}]^\prime}\|b\|_{\mathrm{BMO}} . 
		\end{align*}
		Consequently,
		$$
		\left\|\Pi_l\left(\h, \g_1, \ldots, \g_m\right)\right\|_{\mathbf H^1} \lesssim \|\h\|_{ [X^{p_0}]^\prime}\left\|\g_1\right\|_{X^{p_1}} \ldots\left\|\g_m\right\|_{X^{p_m}},
		$$
		which show that $\Pi_l\left(\h, \g_1, \ldots, \g_m\right)$ is an element of Hardy space $\mathbf H^1(\mathbb R^n)$.
	\end{proof}
	
	\begin{lemma}\label{lem characterization function}
		Let $X$ be a ball Banach function space and Hardy-Littlewood maximal operator $M$ is bounded on $X$. Then, for any $B\in \mathbb B$, we have
		\begin{align*}
			\l\|\chi_B\r\|_{X}\l\|\chi_B\r\|_{X^\prime}=\l|B\r|.
		\end{align*}
	\end{lemma}
	\begin{proof}
		On the one hand, from the H\"older inequality on ball Banach function space, we have
		\begin{align}\label{ine B1}
			\l|B\r|=\int_{\mathbb R^n}\chi_B(x)dx\leq \l\|\chi_B\r\|_{X}\l\|\chi_B\r\|_{X^\prime}.
		\end{align}
		On the other hand, for any $B\subset \mathbb B$ and $f\in L_{\rm loc}^1(\mathbb R^n)$
		\begin{align*}
			\frac{1}{\l|B\r|} \l\|\chi_B\r\|_{X}\l\|\chi_B\r\|_{X^\prime}\leq \sup_{\|g\|_X\leq 1}\frac{\l\|\chi_B\r\|_{X}}{\l|B\r|}\int_{\mathbb R^n} \l|g(x)\r|\chi_B(x)dx.
		\end{align*}
		Set $t=|g|_B=\frac{1}{|B|}\int_{\mathbb R^n} \l|g(x)\r|\chi_B(x)dx$ and $\lambda=t/2$. For any $x\in B$, the following inequality hold
		\begin{align*}
			\lambda=\frac{1}{2|B|}\int_{\mathbb R^n} \l|g(x)\r|\chi_B(x)dx\leq \frac{1}{2}M(g\chi_B)(x).
		\end{align*}
		Since the boundedness of  $M$, the following estimate can be obtained
		\begin{align}\label{ine measureB}
			\frac{\l\|\chi_B\r\|_{X}}{\l|B\r|}\int_{\mathbb R^n} \l|g(x)\r|\chi_B(x)dx&=\l\|\chi_B\r\|_{X} |g|_B \leq \l\|\chi_{\l\{x\in \mathbb R^n:M(g\chi_B)(x)>\lambda\r\}}\r\|_X|g|_B\leq \lambda^{-1}\l\|g\chi_B\r\|_X|g|_B\leq C.
		\end{align}
		Taking the supremum for the two sides of \eqref{ine measureB}, one gets that 
		\begin{align}\label{ine B2}
			\l\|\chi_B\r\|_{X}\l\|\chi_B\r\|_{X^\prime}\lesssim \l|B\r|.
		\end{align}
		Combining \eqref{ine B1} with \eqref{ine B2}, this is the desired conclusion. 
	\end{proof}
	\begin{lemma}\label{lem log}
		Let $1\leq l \leq m$, $1<p_1, \ldots, p_m<\infty$, $ 1/p_0=1/p_1+1/p_2+\dots+1/p_m$ and $1\leq 1/p_0<\infty$. Assume that $f \in \mathbf{H}^1\left(\mathbb{R}^n\right)$ can be represented as $f=\sum_{i = 1}^\infty \lambda_ia_i$. Then, for any atom $a_i(x)(i=1,2,\dots)$, there exists $\h^i\in [X^{p_0}]^\prime$ and $\g_{1}^i\in X^{p_1}, \ldots, \g_{m}^i\in X^{p_m}$  and enough large $M$ such that
		$$
		\left\|a_i-\Pi_l\left(\h^i, \g_1^i, \ldots, \g_m^i\right)\right\|_{\mathbf{H}^1} \lesssim \frac{\log M}{M^\delta},
		$$
		and
		$$
		\left\|\h^i\right\|_{[X^{p_0}]^\prime}\left\|\g_1^i\right\|_{X^{p_1}} \ldots\left\|\g_m^i\right\|_{X^{p_m}} \lesssim M^{m n}.
		$$
		
	\end{lemma}
	
	\begin{proof}
		To study the general case, we write $a_i(x)=a(x)$ for any $1\leq i\leq \infty$.  As function $a(x)$ is atom we know that $\operatorname{supp}(a)\subset B(x_0,r)$ for some $x_0\in \mathbb R^n$ and 
		$$
		\|a\|_{L^{\infty}} \leq r^{-n} \quad \text { and } \quad \int_{\mathbb{R}^n} a(x) d x=0 .
		$$
		For any $1 \leqslant l \leqslant m$, taking the appropriate $y_l \in \mathbb{R}^n$ so that $\left|x_0-y_l\right|=M r$. Similarly to the relation of $x_0$ and $y_l$, we choose $y_1, y_2, \ldots, y_{l-1}, y_{l+1}, \ldots, y_m$ such that
		\begin{align}\label{ine relationship y_j}
			\left|x_0-y_l\right|=\left|y_j-y_l\right|=M r, \quad j=1, \ldots, m, j \neq l .
		\end{align}
		By the above choice, we can  check that  $B_i=B(y_j,r)$ and $B(x_0,r)$ are $(m + 1)-$pairwise disjoint balls.\\ 
		Now, for $T_l^*$ as defined in \eqref{ine T*}, let us set
		$$
		\left\{\begin{array}{l}
			\h(x)=\chi_{B_l}(x), \\
			\g_j(x)=\chi_{B_j}(x), \quad j \neq l, \\
			\g_l(x)=\frac{a(x)}{T_l^*\left(\g_1, \ldots, \g_{l-1}, \h, \g_{l+1}, \ldots, \g_m\right)\left(x_0\right)}.
		\end{array}\right.
		$$
		Combining the definition of $T_l^*$, the $mn$ homogeneity of $\mathrm{T}$ and the choice of the function $\g_1, \ldots, \g_{l-1}, \h, \g_{l+1}, \ldots, \g_m$ as above, we deduce that
		$$
		\left|T_l^*\left(\g_1, \ldots, \g_{l-1}, \h, \g_{l+1}, \ldots, \g_m\right)\left(x_0\right)\right| \geq C M^{-m n} .
		$$
		In fact, for any $j=1, 2, \dots, l-1, l+1, \dots, m$
		\begin{align}\label{ine normj1}
			\l\|\h\r\|_{[X^{p_0}]^\prime}=\l\|\chi_{B_l}\r\|_{[X^{p_0}]^\prime}\quad \text{and} \quad \l\|\g_j\r\|_{X^{p_j}}=\l\|\chi_{B_j}\r\|_{X^{p_j}}.
		\end{align} 
		Furthermore,
		$a$ is an atom, then
		\begin{align}\label{ine normj2}
			\l\|\g_l\r\|_{X^{p_l}}&=\l\|\frac{a}{\l(T^*\r)_l\l(\g_1, \dots, \g_{l-1}, \h, \g_{l+1}, \dots, \g_m\r)(x_0)}\r\|_{X^{p_l}}\lesssim M^{mn}\|a\|_{X^{p_l}}\lesssim M^{mn}\|a\|_{L^\infty}\|\chi_{B(x_0, r)}\|_{X^{p_l}}.
		\end{align}
		By the relation of $y_l$ and $y_j$ for any $j=1, \ldots, m, j \neq l$ in \eqref{ine relationship y_j}, we deduce that for any $1\leq j\leq m$, $B_j\subset (M+1)B_l$. Moreover, repeated application of Lemma \ref{lem characterization function} enables us to get 
		$$
		\|\h\|_{[X^{p_0}]^\prime}\left\|\g_1\right\|_{X^{p_l}} \ldots\left\|\g_m\right\|_{X^{p_m}} \lesssim M^{m n} \l|B(y_l,(M+1)r)\r|^{-1}\l\|\chi_{(M+1)B_j}\r\|_{X^{p_1}} \dots \l\|\chi_{(M+1)B_j}\r\|_{X^{p_m}}\lesssim M^{m n} .
		$$
		Now, we turn to claim that
		$$
		\left\|a-\Pi_l\left(\h, \g_1, \ldots, \g_m\right)\right\|_{\mathbf{H}^1} \lesssim \frac{\log M}{M^\delta} .
		$$
		We need only consider the following two parts
		\begin{align*}
			&\quad a(x)-\Pi_l\left(\h, \g_1, \ldots, \g_m\right)(x)\\
			&= 
			a(x)\left(1-\frac{T_l^*\left(\g_1, \ldots, \g_{l-1}, \h, \g_{l+1}, \ldots, \g_m\right)(x)}{T_l^*\left(\g_1, \ldots, \g_{l-1}, \h, \g_{l+1}, \ldots, \g_m\right)\left(x_0\right)}\right)+\h(x) T\left(\g_1, \ldots, \g_m\right)(x) \\
			&:=\mathbf{P}_1+\mathbf{P}_2,
		\end{align*}
		We write $\vec{z}_{(1, t)}=(z_1, z_2, \dots, z_t)$. Moreover, we conclude from the  smooth condition \eqref{smooth condition} of kernel $K$  that
		\begin{align*}
			\left|\mathbf{P}_1\right|= & |a(x)| \frac{\left|T_l^*\left(\g_1, \ldots, \g_{l-1}, \h, \g_{l+1}, \ldots, \g_m\right)\left(x_0\right)-T_l^*\left(\g_1, \ldots, \g_{l-1}, \h, \g_{l+1}, \ldots, \g_m\right)(x)\right|}{\left|T_l^*\left(\g_1, \ldots, \g_{l-1}, \h, \g_{l+1}, \ldots, \g_m\right)\left(x_0\right)\right|} \\
			\lesssim & M^{m n}\|a\|_{L^{\infty}} \int_{\Pi_{j=1}^m B_j} \mid K\left(\vec{z}_{(1, l-1)}, x_0, \vec{z}_{(l+1, m)}\right) -K\left(\vec{z}_{(1, l-1)}, x, \vec{z}_{(l+1, m)}\right) \mid d \vec{z} \\
			\lesssim & M^{m n} r^{-n} \int_{\Pi_{j=1}^m B_j} \frac{\left|x_0-x\right|^\delta}{\left(\sum_{i=1, i \neq l}^m\left|z_l-z_i\right|+\left|z_l-x_0\right|\right)^{m n+\delta}} d \vec{z} \\
			\lesssim &  M^{-\delta} r^{-n} ,
		\end{align*}
		Since the atom $a$ is support on $B(x_0,r)$, one obtains that
		$$
		\left|\mathbf{P}_1\right| \lesssim M^{-\delta} r^{-n} \chi_{B_{(x_0, r)}} .
		$$
		From the condition of atom $a$ and the  smooth condition \eqref{smooth condition} of kernel $K$,  one deduces that 
		\begin{align*}
			& \left|\mathbf{P}_2\right|=\chi_{B_l}\left|\int_{\Pi_{j=1, j \neq l}^m B_j \times B\left(x_0, r\right)} K\left(x, \vec{z}_{(1,m)}\right) \frac{a\left(z_l\right)}{T_l^*\left(\g_1, \ldots, \g_{l-1}, \h, \g_{l+1}, \ldots, \g_m\right)\left(x_0\right)} d \vec{z}\right| \\
			& \lesssim \chi_{B_l} M^{mn}\left|\int_{\Pi_{j=1, j \neq l}^m B_j \times B\left(x_0, r\right)}\left[K\left(x, \vec{z}_{(1,m)} \right)-K\left(x_0, \vec{z}_{(1,m)}\right)\right] a\left(z_l\right) d \vec{z}\right| \\
			& \lesssim \chi_{B_l} M^{mn} \|a\|_{L^{\infty} }\int_{\Pi_{j=1, j \neq l}^m B_j \times B\left(x_0, r\right)} \frac{\left|x-x_0\right|^\delta}{\left(\sum_{j=1}^m\left|x_0-z_j\right|\right)^{m n+\delta}} d \vec{z} \\
			& \lesssim M^{-\delta}r^{-n}\chi_{B_l}.
		\end{align*}
		According to the two parts $\mathbf {P}_1$ and $\mathbf {P}_2$, we gets that
		$$
		\left|a(x)-\Pi_l\left(\h, \g_1, \ldots, \g_m\right)(x)\right| \lesssim \frac{\left(\chi_{B_{(x_0, r)}}+\chi_{B_l}\right)}{M^{\delta} r^{n}}.
		$$
		Hence, the Lemma \ref{lem log} shows that  
		$$
		\left\|a-\Pi_l\left(\h, \g_1, \ldots, \g_m\right)\right\|_{\mathbf{H}^1} \lesssim  \frac{\log M}{M^\delta} .
		$$
		This proof is proven.
	\end{proof}
	We now prove the main Theorem \ref{the factorization}.
	\begin{proof}[Proof of Theorem \ref{the factorization}]
		Let us first prove that for any $f$ with the representation as \eqref{ine represent}, we have
		\begin{align*}
			\|f\|_{\mathbf{H}^1} \lesssim \inf \left\{\sum_{i=1}^{\infty} \sum_{j=1}^{\infty}\left|\lambda_{i,j}\right|\left\|\h_j^i\right\|_{[X^p]^\prime}\left\|\g_{j, 1}^i\right\|_{X^{p_1}} \dots \l\|\g_{j, m}^i\r \|_{X^{p_m}} \right\}.
		\end{align*} 
		From the Lemma \ref{lem proud_lH1}, the representation \eqref{ine represent} and sequences $\left\{\lambda_{i,j}\right\} \in \ell^1$, we know that
		\begin{align*}
			\|f\|_{\mathbf{H}^1}&=\lim_{n\rightarrow\infty}\lim_{m\rightarrow\infty}\sum_{i=1}^{n} \sum_{j=1}^{m} |\lambda_{i,j}|\l\| \Pi_l\left(\h_j^i, \g_{j, 1}^i, \ldots, \g_{j, m}^i\right)\r\|_{\mathbf{H}^1}\\
			&\leq \sum_{i=1}^{\infty} \sum_{j=1}^{\infty} |\lambda_{i,j}|\|\h_j^i\|_{ [X^{p_0}]^\prime}\left\|\g_{j,1}^i\right\|_{X^{p_1}} \ldots\left\|\g_{j,m}^i\right\|_{X^{p_m}},
		\end{align*}
		which implies that 
		\begin{align}\label{ine onehand}
			\|f\|_{\mathbf{H}^1}\leq \inf \l\{\sum_{i=1}^{\infty} \sum_{j=1}^{\infty} |\lambda_{i,j}|\|\h_j^i\|_{ [X^{p_0}]^\prime}\left\|\g_{j,1}^i\right\|_{X^{p_1}} \ldots\left\|\g_{j,m}^i\right\|_{X^{p_m}}\r\}.
		\end{align}

		We now turn to claim that for any $f\in \mathbf{H}^1(\mathbb R^n)$, there exists a representation as \ref{ine represent}, we have 
		\begin{align*}
			\sum_{i=1}^{\infty} \sum_{j=1}^{\infty}\left|\lambda_{i,j}\right|\left\|\h_j^i\right\|_{[X^p]^\prime}\left\|\g_{j, 1}^i\right\|_{X^{p_1}} \dots \l\|\g_{j, m}^i\r \|_{X^{p_m}} \lesssim 	\|f\|_{\mathbf{H}^1}.
		\end{align*}
		The Definition \ref{def H1} tells us that for any $f\in \mathbf{H}^1(\mathbb R^n)$ can be represented as a linear combination of atoms, i.e.
		$$
		f=\sum_{i= 1}^\infty \lambda_{i,1} a_{i,1} \quad \text{and }\quad \|f\|_{\mathbf{H}^1\left(\mathbb{R}^n\right)}=\inf \left\{\sum_{i = 1}^\infty\left|\lambda_{i,1}\right|: f=\sum_{i = 1}^\infty \lambda_{i,1} a_{i,1}\right\} .
		$$
		In fact, 
		\begin{align*}
			f(x)=\sum_{i=1}^\infty \lambda_{i,1}a_{i,1}(x)=\sum_{i=1}^\infty \lambda_{i,1}\l(a_{i,1}-\Pi_l\left(\h^i_1, \g_{1,1}^i, \ldots, \g_{1,m}^i\right)\r)+\sum_{i = 1}^\infty \lambda_{i,1}\Pi_l\left(\h^i_1, \g_{1,1}^i, \ldots, \g_{1,m}^i\right)=\mathbf{A}_1+\mathbf{B}_1.
		\end{align*}
		We apply the Lemma \ref{lem log} for all atom $a_{i,1}$, then, we know that there exist $\h^i_1\in [X^{p_0}]^\prime$ and $\g_{1,1}^i\in X^{p_1}, \ldots, \g_{1,m}^i\in X^{p_m}$, such that 
		$$
		\left\|a_{i,1}-\Pi_l\left(\h^i_1, \g_{1,1}^i, \ldots, \g_{1,m}^i\right)\right\|_{\mathbf{H}^1} \leq C \frac{\log M}{M^\delta} .
		$$
		Furthermore, for large enough $M$, the following inequalities hold
		\begin{align*}
			\left\|\mathbf{A}_1\right\|_{\mathbf{H}^1}=&\left\|f_1-\sum_{i= 1}^\infty \lambda_{i,1} \Pi_l\left(\h^i_1, \g_{1,1}^i, \ldots, \g_{1,m}^i\right)\right\|_{\mathbf{H}^1} \\
			& \leq \sum_{i= 1}^\infty \left|\lambda_{i,1}\right|\left\|a_{i,1}-\Pi_l\left(\h^i_1, \g_{1,1}^i, \ldots, \g_{1,m}^i\right)\right\|_{\mathbf{H}^1} \\
			&\lesssim  \frac{\log M}{M^\delta} \sum_{i= 1}^\infty \left|\lambda_{i,1}\right| 
			\leq \frac{1}{2}\|f\|_{\mathbf{H}^1},
		\end{align*}
		which implies that $\mathbf{A}_1\in \mathbf{H}^1$.\\
		Therefore,
		\begin{align*}
			\sum_{i =1}^\infty \left|\lambda_{i,1}\right|	\left\|\h^i_1\right\|_{[X^{p_0}]^\prime}\left\|\g_{1,1}^i\right\|_{X^{p_1}} \ldots\left\|\g_{1,m}^i\right\|_{X^{p_m}} \lesssim M^{m n}\|f\|_{\mathbf{H}^1}
		\end{align*}
		We set ~$\mathbf{A}_1=f_1$, then
		\begin{align*}
			f_1=f-\sum_{i= 1}^\infty \lambda_i \Pi_l\left(\h^i, \g_1^i, \ldots, \g_m^i\right) .
		\end{align*}
		Using the similar method, we know that for any $f_1\in \mathbf{H}^1(\mathbb R^n)$, we can find a $\{\lambda_{i,2}\}\in \ell^1$ and sequence of atom $\{a_{i,2}\}$, such that $f$ can be written as
		$$
		f_1=\sum_{i= 1}^\infty \lambda_{i,2} a_{i,2} \quad \text{and }\quad \|f_1\|_{\mathbf{H}^1\left(\mathbb{R}^n\right)}=\inf \left\{\sum_{i = 1}^\infty\left|\lambda_{i,2}\right|: f=\sum_{i = 1}^\infty \lambda_{i,2} a_{i,2}\right\} .
		$$
		Hence,
		\begin{align*}
			f_1(x)=\sum_{i=1}^\infty \lambda_{i,2}a_{i,2}(x)=\sum_{i=1}^\infty \lambda_{i,2}\l(a_{i,2}-\Pi_l\left(\h^i_2, \g_{2,1}^i, \ldots, \g_{2,m}^i\right)\r)+\sum_{i = 1}^\infty \lambda_{i,2}\Pi_l\left(\h^i_2, \g_{2,1}^i, \ldots, \g_{2,m}^i\right)=\mathbf{A}_2+\mathbf{B}_2.
		\end{align*}
		We apply the Lemma \ref{lem log} for all atom $a_{i,2}$, then, we know that there exist $\{\lambda_{i,j}\}\in \ell^1$ and $\h^i_2\in [X^{p_0}]^\prime$ and $\g_{2,1}^i\in X^{p_1}, \ldots, \g_{2,m}^i\in X^{p_m}$, such that 
		$$
		\left\|a_{i,2}-\Pi_l\left(\h^i_2, \g_{2,1}^i, \ldots, \g_{2,m}^i\right)\right\|_{\mathbf{H}^1} \leq C \frac{\log M}{M^\delta} .
		$$
		Furthermore, for large enough $M$, the following inequalities hold
		\begin{align*}
			&\quad \left\|f_1-\sum_{i= 1}^\infty \lambda_{i,2} \Pi_l\left(\h^i_2, \g_{2,1}^i, \ldots, \g_{2,m}^i\right)\right\|_{\mathbf{H}^1} \\
			& \leq \sum_{i= 1}^\infty \left|\lambda_{i,2}\right|\left\|a_{i,1}-\Pi_l\left(\h^i_1, \g_{2,1}^i, \ldots, \g_{2,m}^i\right)\right\|_{\mathbf{H}^1} \\
			&\lesssim  \frac{\log M}{M^\delta} \sum_{i= 1}^\infty \left|\lambda_{i,2}\right| 
			\leq \frac{1}{2^2}\|f\|_{\mathbf{H}^1},
		\end{align*}
		which implies that $\mathbf{A}_2\in \mathbf{H}^1$. \\
		Therefore,
		\begin{align*}
			\sum_{i =1}^\infty \left|\lambda_{i,2}\right|	\left\|\h^i_2\right\|_{[X^{p_0}]^\prime}\left\|\g_{2,1}^i\right\|_{X^{p_1}} \ldots\left\|\g_{2,m}^i\right\|_{X^{p_m}}  \lesssim M^{m n}\|f_1\|_{\mathbf{H}^1}\lesssim \frac{M^{m n}}{2^2}\|f\|_{\mathbf{H}^1}
		\end{align*}
		It is easy to see that 
		\begin{align*}
			f(x)=\mathbf{A}_1+\mathbf{B}_1=\mathbf{A_2}+\mathbf{B}_1+\mathbf{B}_2=\mathbf{A_2}+\sum_{i=1}^\infty \sum_{j=1}^2\lambda_{i,j} \Pi_l\left(\h^i_j, \g_{j,1}^i, \ldots, \g_{j,m}^i\right).
		\end{align*}
		Repeated application the above step for each function $f_i(x)~(1\leq i\leq K)$, then, there exist $\h^i_j\in [X^{p_0}]^\prime$ and $\g_{j,1}^i\in X^{p_1}, \ldots, \g_{j,m}^i\in X^{p_m}$, such that $\mathbf{A}_K\in \mathbf{H}^1$,
		\begin{align*}
			f(x)=\mathbf{A}_K+\sum_{i=1}^\infty \sum_{j=1}^i\lambda_{i,j} \Pi_l\left(\h^i_j, \g_{j,1}^i, \ldots, \g_{j,m}^i\right).
		\end{align*}
		and 
		\begin{align*}
			\sum_{i =1}^\infty \left|\lambda_{i,2}\right|	\left\|\h^i_j\right\|_{[X^{p_0}]^\prime}\left\|\g_{j,1}^i\right\|_{X^{p_1}} \ldots\left\|\g_{j,m}^i\right\|_{X^{p_m}} \lesssim \frac{M^{m n}}{2^i}\|f\|_{\mathbf{H}^1}.
		\end{align*}
		Thus, letting $N \rightarrow \infty$, it follows that
		\begin{align*}
			f(x)=\sum_{i=1}^\infty \sum_{j=1}^\infty \lambda_{i,j} \Pi_l\left(\h^i_j, \g_{j,1}^i, \ldots, \g_{j,m}^i\right)
		\end{align*}
		and 
		\begin{align}\label{ine hand}
			\sum_{i =1}^\infty \sum_{j =1}^\infty \left|\lambda_{i,2}\right|	\left\|\h^i_j\right\|_{[X^{p_0}]^\prime}\left\|\g_{j,1}^i\right\|_{X^{p_1}} \ldots\left\|\g_{j,m}^i\right\|_{X^{p_m}} \lesssim \|f\|_{\mathbf{H}^1}.
		\end{align}
		Combining the inequalities \eqref{ine onehand} and \eqref{ine hand}, yields
		\begin{align*}
			\|f\|_{\mathbf{H}^1}\approx \inf \l\{\sum_{i=1}^{\infty} \sum_{j=1}^{\infty} |\lambda_{i,j}|\|\h_j^i\|_{ [X^{p_0}]^\prime}\left\|\g_{j,1}^i\right\|_{X^{p_1}} \ldots\left\|\g_{j,m}^i\right\|_{X^{p_m}}\r\}.
		\end{align*}
		This proof is proven.
	\end{proof}

	\begin{proof}[Proof of Theorem \ref{the characterization BMO}]
		On the one hand, the boundedness of $T$ and $[b,T]_l$ implies that the upper bound holds.
		
		On the other hand, we first write that $ b_L(x)=b(x) \mathbf{1}_{B\left(x_0, L\right)}(x) $ for any $b\in L_{ {\rm loc }}^1$. Thanis to the fact that $\mathbf{H}^1\left(\mathbb{R}^n\right) \cap L_c^{\infty}\left(\mathbb{R}^n\right)$ is dense in $\mathbf{H}^1\left(\mathbb{R}^n\right)$.
		The Theorem \ref{the factorization} show that there exist sequences $\left\{\lambda_{i,j}\right\} \in \ell^1$ and functions $\h^i_j\in [X^{p_0}]^\prime$ and $\g_{j,1}^i\in X^{p_1}, \ldots, \g_{j,m}^i\in X^{p_m}$, such that  for any $f \in \mathbf{H}^1\left(\mathbb{R}^n\right) \cap L_c^{\infty}\left(\mathbb{R}^n\right)$, we get
		$$
		f=\sum_{i=1}^{\infty} \sum_{j=1}^{\infty} \lambda_j^i \Pi_l\left(\h_j^i, \g_{j, 1}^i, \ldots, \g_{j, m}^i\right) .
		$$
		At the same time, we have
		\begin{align*}
			\|f\|_{\mathbf{H}^1}\approx \sum_{i=1}^{\infty} \sum_{j=1}^{\infty} |\lambda_{i,j}|\|\h_j^i\|_{ [X^{p_0}]^\prime}\left\|\g_{j,1}^i\right\|_{X^{p_1}} \ldots\left\|\g_{j,m}^i\right\|_{X^{p_m}}.
		\end{align*}
		It is obvious that $\lim_{L\rightarrow \infty} b_L=b$, then, for any $f \in \mathbf{H}^1\left(\mathbb{R}^n\right) \cap L_c^{\infty}\left(\mathbb{R}^n\right)$, one concludes that 
		\begin{align*}
			\langle b, f\rangle=\lim _{L \rightarrow \infty}\left\langle b_L, f\right\rangle
			& =\sum_{i=1}^{\infty} \sum_{j=1}^{\infty} \lambda_{i,j} \lim _{L \rightarrow \infty}\left\langle b_L, \Pi_l\left(\h_j^i, \g_{j, 1}^i, \ldots, \g_{j, m}^i\right)\right\rangle \\
			& =\sum_{i=1}^{\infty} \sum_{j=1}^{\infty} \lambda_{i,j}\left\langle b, \Pi_l\left(\h_j^i, \g_{j, 1}^i, \ldots, \g_{j, m}^i\right)\right\rangle \\
			&=\sum_{i=1}^{\infty} \sum_{j=1}^{\infty} \lambda_{i,j} \int[b, T]_l\left(\g_{j, 1}^i, \ldots, \g_{j, m}^i\right)(x) \h_j^i(x) d x .
		\end{align*}
		Since $[b, T]_l$ maps $X^{p_1}\times \dots \times X^{p_m} \rightarrow X^{p_0}$, and the Lemma \ref{lem log}, we conclude that
		\begin{align*}
			& |\langle b, f\rangle| \leq \sum_{i=1}^{\infty} \sum_{j=1}^{\infty}\left|\lambda_{i,j}\right|\left\|[b, T]_l\left(\g_{j, 1}^i, \ldots, \g_{j, m}^i\right)\right\|_{X^{p_0}}\left\|\h_j^i\right\|_{\l[X^{p_0}\r]^\prime} \\
			& \lesssim\left\|[b, T]_l\right\|_{X^{p_1}\times \dots \times X^{p_m} \rightarrow X^{p_0}}  \sum_{i=1}^{\infty} \sum_{j=1}^{\infty} |\lambda_{i,j}|\|\h_j^i\|_{ [X^{p_0}]^\prime}\left\|\g_{j,1}^i\right\|_{X^{p_1}} \ldots\left\|\g_{j,m}^i\right\|_{X^{p_m}} \\
			&\lesssim\left\|[b, T]_l\right\|_{X^{p_1}\times \dots \times X^{p_m} \rightarrow X^{p_0}} \|f\|_{\mathbf{H}^1} ,
		\end{align*}
		which implies that the proof of this theorem is completed.
	\end{proof}
	\section{Some Examples}
	We will apply the factorizations theorem proved in the previous section to some classical spaces including ball Banach function spaces such as weighted Lebesgue spaces, variable Lebesgue spaces, mixed-norm Lebesgue spaces, etc., further enriching and developing the work about characterization on $H^1(\mathbb R^n)$ via operators on different spaces. In fact, the above spaces have been proved to be ball Banach function spaces, the details can be found in \cite{sawano2017hardy}. 
	
	\subsection{Weighted Lebesgue Spaces}
	Muckenhoupt weights are crucial in characterizing the boundedness of some integral operators on weighted Lebesgue spaces, to cope with the multilinear operators, Lerner et al. introduced multiple weights in \cite{lerner2009new}. Let's first recall the definition of weighted Lebesgue spaces and multiple weights.
	
	Let $1 \leqslant p_1, \ldots, p_m<\infty$. Given $\vec{w}=\left(w_1, \ldots, w_m\right)$, set $
	v_{\vec{w}}=\prod_{i=1}^m w_i^{p / p_i}.$
	The multiple weights $\vec{w}\in A_{\vec{P}}$ if
	$$
	\sup _Q\left(\frac{1}{|Q|} \int_Q v_{\vec{w}}\right)^{1 / p} \prod_{i=1}^m\left(\frac{1}{|Q|} \int_Q w_i^{1-p_i^{\prime}}\right)^{1 / p_i^{\prime}}<\infty.
	$$
	When $p_i=1,\left(\frac{1}{|Q|} \int_Q w_i^{1-p_i^{\prime}}\right)^{1 / p_i^{\prime}}$ is seen as $\left(\inf _Q w_i\right)^{-1}$.
	\begin{theorem}\label{the 1}
		Let $1<p, p_1, \ldots, p_m<\infty, 1/p_=1/p_1+1/p_2+\dots+1/p_m$ and $\vec{w} \in A_{\vec{P}}$. Then for any $f \in$ $\mathbf H^1\left(\mathbb{R}^n\right)$, there exists sequences $\left\{\lambda_{j,i}\right\} \in \ell^1$ and functions $\h_j^i \in L^{p^{\prime}}\left(\nu_{\vec{w}}^{1-p^{\prime}}\right), \g_{j, 1}^i \in L^{p_1}\left(w_1\right), \ldots, \g_{j, m}^i \in$ $L^{p_m}\left(w_m\right)$ such that $f$ can be represented as \eqref{ine represent}. Moreover, the equivalence of the norm holds
		$$
		\|f\|_{\mathbf{H}^1} \approx \inf \left\{\sum_{i=1}^{\infty} \sum_{j=1}^{\infty}\left|\lambda_{i,j}\right|\left\|\h_j^i\right\|_{L^{p^{\prime}}\left(\nu_{\vec{w}}^{1-p^{\prime}}\right)}\left\|\g_{j, 1}^i\right\|_{L^{p_1}\left(w_1\right)} \dots \l\|\g_{j, m}^i\r \|_{L^{p_m}\left(w_m\right)}\right\}.
		$$
	\end{theorem}
	\begin{theorem}\label{the 2}
		Let $1 \leq l \leq m, 1<p, p_1, \ldots, p_m<\infty, 1/p_=1/p_1+1/p_2+\dots+1/p_m$ and $\vec{w} \in A_{\vec{P}}$. Then, the following propositions are equivalent
		\begin{enumerate*}
			\item[\rm (1)] $b \in {\rm BMO}\left(\mathbb{R}^n\right)$;
			\item[\rm (2)] the commutator $[b, T]_l$ is bounded from $L^{p_1}(w_1)\times L^{p_2}(w_2)\times \dots \times L^{p_m}(w_m)$ to $L^{p}(v_{\vec{w}})$.
		\end{enumerate*}
	\end{theorem}
	
	The proofs of Theorem \ref{the 1} and Theorem \ref{the 2} can be obtained by Corollary 3.8 and Corollary 3.9 in \cite{lerner2009new}, Theorem \ref{the factorization} and Theorem \ref{the characterization BMO} in this paper, hence, we omit the details.
	\begin{remark}\rm
		Wang et al. in \cite{wang2023factorization} proved the Theorem \ref{the 1} and \ref{the 2}, which can reduce to the results in \cite{li2017weak} by Li et al. on Lebesgue space. In addition, He et al. \cite{he2023factorization} characterized the weighted Hardy spaces via the multilinear Calder\'on-Zygmund operator on weighted Lebesgue spaces.
	\end{remark}
	
	\subsection{Variable Lebesgue Spaces}
	The variable Lebesgue space $L^{p(\cdot)}(\mathbb R^n)$ was first introduced by Kov{\'a}{\v{c}}ik and R{\'a}kosn{\'\i}k \cite{kovavcik1991spaces} for solving problems related to variational integrals in physics, we first recall its definition.
	
	We define the variable Lebesgue space $L^{p(\cdot)}\left(\mathbb{R}^n\right)$ is the set as follows,
	\begin{align*}
		L^{p(\cdot)}\left(\mathbb{R}^n\right) =\l\{f: \|f\|_{L^{p(\cdot)}}:=\inf \left\{\lambda: \int_{\mathbb{R}^n}\left[\frac{|f(x)|}{\lambda}\right]^{p(x)} d x \leq 1\right\}<\infty \r\},
	\end{align*}
	where $\lambda\in (0,\infty)$ and the function $p(\cdot): \mathbb{R}^n \rightarrow(0, \infty)$ be a measurable function. Furthermore, the $p^{\prime}(\cdot)$ denotes the conjugate exponent of $p(\cdot)$, which also tells us that the dual space of  $L^{p(\cdot)}\left(\mathbb{R}^n\right)$ is $L^{p(\cdot)^\prime}(\mathbb R^n)$. 
	
	We call the $p(\cdot)\in \mathcal{P}\left(\mathbb{R}^n\right)$, if the function $p(\cdot)$ satisfies that for any $x\in \mathbb R^n$ such that
	$$
	p_{-}:=\inf _{x \in \mathbb{R}^n} p(x)>1, \quad p_{+}:=\sup _{x \in \mathbb{R}^n} p(x)<\infty .
	$$
	
	Moreover, we define the $\mathcal B(\mathbb R^n)$ as Theorem 1.1 in \cite{cruz2006boundedness}, which ensures that the maximal operators is bounded on $L^{p(\cdot)}\left(\mathbb{R}^n\right)$. Especially, the Corollary 2.1 and the Corollary 2.2 in \cite{huang2010multilinear} is established by Huang and Xu, which provides boundedness of multilinear operator and yields a condition about the following theorems hold.
	
	\begin{theorem}\label{the 3}
		Let $p(\cdot)\in \mathcal{P}(\mathbb R^n)$, $p_i(\cdot)\in \mathcal{B}(\mathbb R^n)$ for $i=1, 2,\dots, m$ and $1/p(\cdot)=1/p_1(\cdot)+1/p_2(\cdot)+\dots+1/p_m(\cdot)$. Then for any $f \in$ $\mathbf H^1\left(\mathbb{R}^n\right)$, there exists sequences $\left\{\lambda_{j,i}\right\} \in \ell^1$ and functions $\h_j^i \in L^{p^{\prime}(\cdot)}, \g_{j, 1}^i \in L^{p_1(\cdot)}, \ldots, \g_{j, m}^i \in L^{p_m(\cdot)}$ such that $f$ can be represented as \eqref{ine represent}. Moreover, the equivalence of the norm holds
		$$
		\|f\|_{\mathbf{H}^1} \approx \inf \left\{\sum_{i=1}^{\infty} \sum_{j=1}^{\infty}\left|\lambda_{i,j}\right|\left\|\h_j^i\right\|_{L^{p^{\prime}(\cdot)}}\left\|\g_{j, 1}^i\right\|_{L^{p_1(\cdot)}} \dots \l\|\g_{j, m}^i\r \|_{L^{p_m(\cdot)}}\right\}.
		$$
	\end{theorem}
	\begin{theorem}\label{the 4}
		Let $1 \leq l \leq m$, $p(\cdot)\in \mathcal{P}(\mathbb R^n)$, $p_i(\cdot)\in \mathcal{B}(\mathbb R^n)$ for $i=1, 2,\dots, m$ and $1/p(\cdot)=1/p_1(\cdot)+1/p_2(\cdot)+\dots+1/p_m(\cdot)$. Then, the following propositions are equivalent
		\begin{enumerate*}
			\item[\rm (1)] $b \in {\rm BMO}\left(\mathbb{R}^n\right)$;
			\item[\rm (2)] the commutator $[b, T]_l$ is bounded from $L^{p_1(\cdot)}\times \dots \times L^{p_m(\cdot)}$ to $L^{p(\cdot)}$.
		\end{enumerate*}
	\end{theorem}
	\begin{remark}\rm
		The weak factorizations of Hardy spaces on the variable Lebesgue spaces have been established by Wang in \cite{wang2019notes}. In fact, the Theorem \ref{the 3} and Theorem \ref{the 4} return to the results in \cite{wang2019notes} when $m=1$.
	\end{remark}
	\subsection{Herz Spaces}
	We say a function $f$ in the homogeneous Herz space $\dot{K}_q^{\alpha, p}\left(\mathbb{R}^n\right)$, if $f\in L_{\text {loc }}^q\left(\mathbb{R}^n \backslash\{0\}\right)$ satisfies that
	$$
	\|f\|_{\dot{K}_q^{\alpha, p}}:=\left\{\sum_{k \in \mathbb{Z}} 2^{k p \alpha}\left\|f \chi_k\right\|_{L^q}^p\right\}^{1 / p}<\infty.
	$$
	
	The study of Herz spaces has continued since the 1970s, then, Lu and Yang defined Herz spaces, and at the same time systematically studied the characterization and theory of operator on Herz spaces and Herz-Hardy spaces, the details refer to \cite{lu2008herz}. After that, the theory of operators and the theory of spaces on Herz space are becoming increasingly rich. Hence, Theorem 3 and Theorem 4 in \cite{hussain2014multilinear} ensure the following theorems hold.
	
	\begin{theorem}\label{the 5}
		Let $\alpha_i\in \mathbb R$, $1< q_i<\infty$ and $0<p_i<\infty$ for $i=0, 1, \dots, m$ satisfy $-1/q_i<\alpha_i<-1/q_i^{\prime}$. Suppose that $1/p_0=1/p_1+1/p_2+\dots+1/p_m$, $1/q_0=1/q_1+1/q_2+\dots+1/q_m$ and $1/\alpha_0=1/\alpha_1+1/\alpha_2+\dots+1/\alpha_m$. Then, for any $f \in$ $\mathbf H^1\left(\mathbb{R}^n\right)$, there exists sequences $\left\{\lambda_{j,i}\right\} \in \ell^1$ and functions $\h_j^i \in \dot K_{q_0^\prime}^{-\alpha_0,p_0^\prime}, \g_{j, 1}^i\in \dot K_{q_1}^{\alpha_1,p_1}, \ldots, \g_{j, m}^i \in  \dot K_{q_m}^{\alpha_m,p_m}$ such that $f$ can be represnted as \eqref{ine represent}. Furthermore,
		$$
		\|f\|_{\mathbf{H}^1} \approx \inf \left\{\sum_{i=1}^{\infty} \sum_{j=1}^{\infty}\left|\lambda_{i,j}\right|\left\|\h_j^i\right\|_{\dot K_{q_0^\prime}^{-\alpha_0,p_0^\prime}}\left\|\g_{j, 1}^i\right\|_{\dot K_{q_1}^{\alpha_1,p_1}} \dots \l\|\g_{j, m}^i\r \|_{\dot K_{q_m}^{\alpha_m,p_m}}\right\}.
		$$
	\end{theorem}
	\begin{theorem}\label{the 6}
		Let $1 \leq l \leq m$,  $\alpha_i\in \mathbb R$, $1< q_i<\infty$ and $0<p_i<\infty$ for $i=0, 1, \dots, m$ satisfy $-1/q_i<\alpha_i<-1/q_i^{\prime}$. Suppose that $1/p_0=1/p_1+1/p_2+\dots+1/p_m$, $1/q_0=1/q_1+1/q_2+\dots+1/q_m$,$1/\alpha_0=1/\alpha_1+1/\alpha_2+\dots+1/\alpha_m$. Then, the following propositions are equivalent
		\begin{enumerate*}
			\item[\rm (1)] $b \in {\rm BMO}\left(\mathbb{R}^n\right)$;
			\item[\rm (2)] the commutator $[b, T]_l$ is bounded from $	\dot{K}_{q_1}^{\alpha_1, p_1}\times \dots \times \dot{K}_{q_m}^{\alpha_m, p_m}$ to	$\dot{K}_{q_0}^{\alpha_0, p_0}$.
		\end{enumerate*}
	\end{theorem}
	\subsection{Mixed-norm Lebesgue Spaces}
	Let $0<q_i\leq \infty$ for all $i=1, 2, \dots$. The mixed-norm Lebesgue space $L^{\vec{p}}\left(\mathbb{R}^n\right)$ is defined by
	\begin{align*}
		L^{\vec{p}}(\mathbb R^n)=\l\{f\in \mathcal M: \|f\|_{\vec{p}}=\left\{\int_{\mathbb{R}} \cdots\left[\int_{\mathbb{R}}\left|f\left(x_1, \ldots, x_n\right)\right|^{p_1} d x_1\right]^{\frac{p_2}{p_1}} \cdots d x_n\right\}^{\frac{1}{p_n}}<\infty \r\}.
	\end{align*}
	
	Benedek and Panzone \cite{benedek1961space} modified Lebesgue space to mixed-norm Lebesgue space by improving the norm in each dimension in 1961. Recently, this space has been extensively studied again since it has a better structure that is often applied to solving space-time partial differential equations. Here, the conditions of boundedness in Theorem \ref{the factorization} can be established by the extrapolation on ball Banach function space in \cite{deng2023extrapolations}. Hence, we gain the corresponding weak factorizations theorem on mixed-norm Lebesgue as follows.
	
	\begin{theorem}\label{the 5}
		Let $1<\vec{p}<\infty$. Then, for any $f \in$ $\mathbf H^1\left(\mathbb{R}^n\right)$, there exists sequences $\left\{\lambda_{j,i}\right\} \in \ell^1$ and functions $\h_j^i \in L^{\vec{p}^\prime}~\text{and}~ \g_{j}^i\in L^{\vec{p}}$ such that $f$ can be represented as \eqref{ine represent}. Moreover, 
		$$
		\|f\|_{\mathbf{H}^1} \approx \inf \left\{\sum_{i=1}^{\infty} \sum_{j=1}^{\infty}\left|\lambda_{i,j}\right|\left\|\h_j^i\right\|_{L^{\vec{p}^\prime}}\left\|\g_{j}^i\right\|_{L^{\vec{p}}} \right\}.
		$$
	\end{theorem}
	\begin{theorem}\label{the 6}
		Let $1<\vec{p}<\infty$. Then, the commutator $[b, T]$ is bounded on $L^{\vec{p}}$  if and only if  $b \in {\rm BMO}\left(\mathbb{R}^n\right)$.
	\end{theorem}
	
	\subsection{Lorentz Spaces}
	Before recalling the definition of the Lorent space, we give the definition of the decreasing rearrangement function. We define the decreasing rearrangement of measurable function $f$ by
	$$
	f^*(t)=\inf \left\{s>0: d_f(s) \leq t\right\},
	$$
	where
	$$
	d_f(s):=\mu(\{x \in X:|f(x)|>s\}), \quad \forall s>0 .
	$$
	
	Let $0<p, r \leq$ $\infty$. The Lorentz space is defined to be the set  as follows
	\begin{align*}
		L^{p, r}(\mathbb R^n)=\l\{f\in \mathcal M: \|f\|_{L^{p, r}}<\infty\r\},
	\end{align*}
	here, the norm $\|f\|_{L^{p, r}}$ is defined by
	\begin{align*}
		\|f\|_{L^{p, r}}=\left(\int_0^{\infty}\left(t^{\frac{1}{p}} f^*(t)\right)^r \frac{d t}{t}\right)^{\frac{1}{r}}~\text{if}~r<\infty ~~~~\text{and }~~~~ \|f\|_{L^{p, r}}=\sup _{t>0} t^{\frac{1}{p}} f^*(t) ~ \text { if } r=\infty.
	\end{align*}
	
	Lorentz space can reduce to Lebesgue space, that is $L^{p,p}(\mathbb{R}^n)=L^p(\mathbb R^n)$. For more details of Lorentz space,  we can see the book \cite[the Section 1.4]{2014Classical}, in which, we know the dual space of $L^{p, r}(\mathbb R^n)$ is $L^{p^\prime, r^\prime}(\mathbb R^n)$. Tao et al.\cite{tao2023boundedness} established the boundedness of multilinear operators on Lorentz space, which ensures the following theorem holds.
	
	\begin{theorem}\label{the 5}
		Let $1<p<\infty, r \in[1, \infty]$. Then, for any $f \in$ $\mathbf H^1\left(\mathbb{R}^n\right)$, there exists sequences $\left\{\lambda_{j,i}\right\} \in \ell^1$ and functions $\h_j^i \in L^{p^\prime, r^\prime}, \g_{j}^i\in L^{p,r}$ such that $f$ can be represented as \eqref{ine represent}. Moreover, we have that
		$$
		\|f\|_{\mathbf{H}^1} \approx \inf \left\{\sum_{i=1}^{\infty} \sum_{j=1}^{\infty}\left|\lambda_{i,j}\right|\left\|\h_j^i\right\|_{L^{p^\prime, r^\prime}}\left\|\g_{j}^i\right\|_{L^{p,r}} \right\}.
		$$
	\end{theorem}
	\begin{theorem}\label{the 6}
		Let $1<p<\infty, r \in[1, \infty]$. Then, the commutator $[b, T]$ is bounded on $L^{p,r}$  if and only if  $b \in {\rm BMO}\left(\mathbb{R}^n\right)$.
	\end{theorem}
	

	\noindent
	\textbf{\bf Funding information}\\
	The research was supported by National Natural Science Foundation of China (Grant Nos. 12061069 and 12271483).\\
	\textbf{\bf Contributions}\\
	All authors have reviewed the manuscript.\\
	\textbf{\bf Conflict of interest}\\
     The authors declare that they have no conflict of interest.

	\noindent Yichun Zhao and Jiang Zhou\\
	\medskip
	\noindent
	College of Mathematics and System Sciences\\
	Xinjiang University\\
	Urumqi 830046\\
	\smallskip
	\noindent{E-mail }:
	\texttt{zhaoyichun@stu.xju.edu.cn} (Yichun Zhao);
	\texttt{zhoujiang@xju.edu.cn} (Jiang Zhou)\\
	
	\bigskip \noindent
	Xiangxing Tao\\
	\medskip
	\noindent 
	Department of Mathematics\\
	Zhejiang University of Science and Technology\\
	Hangzhou 310023\\
	\smallskip
	\noindent{E-mail }:
	\texttt{xxtao@zust.edu.cn} (Xiangxing Tao)\\
	\medskip
	
\end{document}